\documentclass[letterpaper, 10 pt, conference]{ieeeconf}	

\IEEEoverridecommandlockouts    

\usepackage{hyperref}
\usepackage{graphicx,color}
\graphicspath{{./IMG/}{images/}}
\usepackage{amsmath}
\usepackage{amssymb}
\usepackage{mathrsfs}
\usepackage{subfigure}
\usepackage{url}
\usepackage{booktabs}
\usepackage{array}
\usepackage[table]{xcolor}
\usepackage{mathtools} 		
\usepackage{epstopdf}
\usepackage{cite}
\usepackage[vlined,ruled]{algorithm2e}
\usepackage{upgreek}
\usepackage{dsfont}
\usepackage{tikz}
\usetikzlibrary{calc}
\usepackage{svg}
\usepackage{grffile}
\usepackage[utf8]{inputenc}
\DeclareUnicodeCharacter{22C6}{\nobreakspace}
\usepackage{pdfpages}
\usepackage{lipsum}
\usepackage{caption}
\usepackage{floatrow}   
\usepackage{subcaption}
\usepackage[framemethod=tikz]{mdframed}
\usepackage{subfigure}

\providecommand{\com[2]}{\begin{tt}[#1: #2]\end{tt}}

\newtheorem{theorem}{Theorem}[section]
\newtheorem{lemma}[theorem]{Lemma}

\newtheorem{corollary}[theorem]{Corollary}
\newtheorem{proposition}[theorem]{Proposition}

\DeclareMathOperator*{\argmin}{argmin}

\DeclareMathOperator*{\st}{s.t.}

\title{On the Set of Possible Minimizers of a Sum of Convex Functions}

\author{Moslem Zamani$^{1}$,  Fran\c{c}ois Glineur$^{1,2}$, and Julien M. Hendrickx$^{1}$
\thanks{$^{1}$ ICTEAM/INMA / $^{2}$ CORE, Université catholique de Louvain, B-1348 Louvain-la-Neuve, Belgium. 
       {\tt\scriptsize \{moslem.zamani, francois.glineur, julien.hendrickx\}@uclouvain.be}}%
}

\begin{document}
\maketitle

\begin{abstract}
Consider a sum of convex functions, where the only information known about each individual summand is the location of a minimizer. In this work, we give an exact characterization of the set of possible minimizers of the sum.

Our results cover several types of assumptions on the summands, such as smoothness or strong convexity. Our main tool is the use of necessary and sufficient conditions for interpolating the considered function classes, which leads to shorter and more direct proofs in comparison with previous work.

We also address the setting where each summand minimizer is assumed to lie in a unit ball, and prove a tight bound on the norm of any minimizer of the sum. 

\end{abstract}

\section{Introduction}
Consider the following optimization problem 
\begin{equation}\label{P}
\textstyle \min_{x\in\mathbb{R}^n} \ f(x):=\textstyle\sum_{i\in [m]} f_i(x),
\end{equation}
where $f_i:\mathbb{R}^n\to (-\infty, \infty]$, $i\in [m]$, are convex functions. Problem \eqref{P} appears naturally in a variety of settings, including distributed optimization; see \cite{nedic2018distributed, yang2019survey} and references therein.  In some situations including resilient distributed optimization \cite{sundaram2018distributed} and some federated learning setups \cite{konevcny2016federated}, functions $f_i$'s are unknown and the only knowledge about each summand is the location of a  minimizer $x_i^\star$, that is, 
\[ x_i^\star \in\argmin_{x\in\mathbb{R}^n} f_i(x), \ \ \ \
i\in [m]. \]
\noindent The problem we consider is to determine the set of potential 
minimizers of problem \eqref{P} using only those known individual summand minimizers $x_i^\star$ and some additional properties of the functions $f_i$, such as smoothness or strong convexity.

\smallskip In a recent line of work, Kuwaranancharoen and Sundaram \cite{kuwaranancharoen2018location, kuwaranancharoen2020set, kuwaranancharoen2023minimizer} study this problem for $m=2$ differentiable and strongly convex functions, with the additional knowledge of an upper bound on the norm of the gradient of each summand at the minimizer $x^\star$ (see Proposition \eqref{Prop.3} for our result in a similar setting). Using a geometric approach, they provide a near-exact characterization of the set of possible minimizers, using \emph{ad hoc} coordinates (more precisely, \cite[Theorem 6.2]{kuwaranancharoen2023minimizer} characterizes the boundary of the region of potential minimizers, which coincides with its interior). Assuming that the summands are smooth (i.e. that their gradient is Lipschitz) is also mentioned in \cite{kuwaranancharoen2023minimizer} as a possibility for further work.


Hendrickx and Rabbat consider a sum of smooth strongly convex functions in the context of 
open multi-agent systems \cite{{hendrickx2020stability}}. They assume all summand minimizers lie in a ball centered at the origin, and derive an upper bond for the norm of a potential minimizer of the sum, expressed in terms of the condition number of the summands \cite[Theorem 1]{hendrickx2020stability}.

In this paper, we show how the use of the interpolation conditions \cite{taylor2017smooth,rubbens2023interpolation} described in Section~\ref{Sec.I}, provides a more direct and algebraic approach to tackle these problems. We give in Section~\ref{Sec.1} an exact characterization of the set of potential minimizers of a sum of two smooth strongly convex functions, and we provide an alternative, simpler characterization of set of possible minimizers of a sum of two strongly convex functions in the setting of \cite{kuwaranancharoen2023minimizer}.

In Section~\ref{Sec.2}, we extend our results in a straightforward way to sums involving arbitrary numbers of functions, and to a variation of the problem where all but one function is smooth.
Finally we provide in Section~\ref{Sec.J} a tight bound for the setting where each summand minimizer lies in a ball.



\section{Preliminaries and Interpolation Conditions}\label{Sec.I}

\subsection{Notations and Definitions}

Throughout the paper, $\|\cdot\|$  and $\langle\cdot,\cdot\rangle$ denote the Euclidean norm and the dot product, respectively. We also use iff as an abbreviation for if and only if. The set of the first $m$ natural numbers is denoted by $[m]$, and $[\ell, m]$ denotes $[m] \setminus [\ell-1]$.

Let $f:\mathbb{R}^n\to(-\infty, \infty]$ be an extended real valued convex function. Function $f$ is proper if its effective domain, defined as $\mathrm{dom}\ f = \{x: f(x)<\infty\}$, is non-empty. It is closed if its epigraph $\{(x, r): f(x)\leq r\}$ is a closed subset of $\mathbb{R}^{n+1}$. The set of subgradients at $x \in \mathrm{dom}\ f$ (subdifferential) is
$$
\partial f(x)=\{g: f(y)\geq f(x)+\langle g, y-x \rangle, \forall y\in\mathbb{R}^n\}. $$
 
\noindent Let $L\in(0, \infty]$ and $\mu\in [0, \infty)$. A convex function $f:\mathbb{R}^n\to(-\infty, \infty]$ is said to be $L$-smooth if for any $x_1, x_2\in\mathbb{R}^n$,
$$
\|g_1-g_2\|\leq L\|x_1-x_2\| \ \ \forall  g_1\in\partial f(x_1),\  g_2\in\partial f(x_2).
$$
If $L<\infty$, then $f$ must be differentiable on $\mathbb{R}^n$. In addition, any convex function is $\infty$-smooth.
 Furthermore, a function $f:\mathbb{R}^n\to(-\infty, \infty]$ is said to be $\mu$-strongly convex if
 the function $x \mapsto f(x)-\tfrac{\mu}{2}\| x\|^2$ is convex.
Note that any convex function is $0$-strongly convex. 
  We denote the set of closed proper convex functions that are $L$-smooth and $\mu$-strongly convex by $\mathcal{F}_{\mu,L}(\mathbb{R}^n)$. In addition, $\mathcal{M}_{\mu,L}(x^\star)$ stands for the set of functions in $\mathcal{F}_{\mu,L}(\mathbb{R}^n)$ for which $x^\star$ is a minimizer. We make the convention $\tfrac{a}{\infty} = 0$ for $a\in\mathbb{R}$.

\subsection{Interpolation Results}

Interpolation constraints are necessary and sufficient conditions for a finite set of vectors and values to be consistent with an actual function in a given function class. For example, we will heavily rely on the following interpolation result for the class of smooth strongly convex functions.

\smallskip 
\begin{theorem}\label{sscint} (see \cite[Theorem 4]{taylor2017smooth}) Consider the class of functions $\mathcal{F}_{\mu,L}(\mathbb{R}^n)$ with $0\leq\mu< L\leq \infty$. Given a set of triplets $\{(x_i; g_i; f_i)\}_{i\in [m]}$ with $x_i,g_i\in\mathbb{R}^n$ and $f_i\in \mathbb{R}$, there exists a function $f\in\mathcal{F}_{\mu,L}(\mathbb{R}^n)$ satisfying 
$$
f(x_i)=f_i, \ \ g_i\in\partial f(x_i),\ i\in [m],
$$
i.e.\@ a function that interpolates the function and (sub)gradient values as prescribed in $(x_i; g_i; f_i)$, {if and only if} the following set of interpolating conditions is satisfied:
\begin{align}\label{int_m}
\nonumber &f_i-f_j-\langle g_j, x_i-x_j\rangle\geq (1-\tfrac{\mu}{L})^{-1} \big(\tfrac{1}{2L}\|g_i-g_j\|^2+ \\
& \ \  \tfrac{\mu}{2}\|x_i-x_j\|^2-\tfrac{\mu}{L}\langle g_i-g_j, x_i-x_j\rangle\big), \ \ \ \forall i, j\in [m].
\end{align}
\end{theorem}

\smallskip Interpolation conditions provide 
an exact characterization of the triplets $\{(x_i; g_i; f_i)\}_{i\in [m]}$ that are consistent with actual functions, and as such, 
are instrumental in the automated tight analysis of black-box optimization algorithm \cite{taylor2017smooth}.

\medskip Our analysis will rely on the following central Lemma, which provides an interpolation condition for a minimizer $x^\star$ and an arbitrary point $x$, only involving those two points and their gradients, but not their function values. 

\smallskip \begin{lemma}\label{Lem1}
Consider the function class $\mathcal{F}_{\mu,L}(\mathbb{R}^n)$ with $0 \le \mu < L \le \infty$. Given two points $x, x^\star \in\mathbb{R}^n$ and a (sub)gradient $g\in\mathbb{R}^n$, there exists a function $f\in\mathcal{M}_{\mu,L}(x^\star)$ with  $g\in\partial f(x)$ if and only if 
\begin{align}\label{int}
& \langle g, x-x^\star\rangle \geq ({1+\tfrac{\mu}{L}})^{-1}\left(\tfrac{1}{L} \left\| g \right\|^2+\mu \| x-x^\star \|^2  \right).
\end{align}
\end{lemma}

\smallskip \begin{proof}
The existence of a function $f$ satisfying the assumptions in the Lemma is equivalent to asking that the two triplets $\{(x;g;f_x),(x^\star;0;f^\star) \}$ are interpolable by a function in $\mathcal{F}_{\mu,L}(\mathbb{R}^n)$ for some values of $f_x$ and $f^\star$. Indeed, we have $f_x = f(x)$, $f^\star = f(x^\star)$ and the optimality condition for a minimizer of a convex function gives $g^\star = 0\in\partial f(x^\star)$. 
Theorem~\ref{sscint} states that it is equivalent to the following pair of conditions \eqref{int_m}, first for $(x_i,x_j) = (x,x^\star)$ then for $(x^\star,x)$: \begin{align*} f_x - f^\star  & \geq  (1-\tfrac{\mu}{L})^{-1} E(x,g,x^\star) \\ f^\star - f_x -  \langle g, x^\star-x \rangle & \geq  (1-\tfrac{\mu}{L})^{-1} E(x,g,x^\star)  \end{align*} 
where $E(x,g,x^\star) = \tfrac{1}{2L}\|g \|^2+ \tfrac{\mu}{2}\|x-x^\star\|^2-\tfrac{\mu}{L}\langle g, x-x^\star \rangle$.
Hence we need to show that these two conditions holds for some $f_x$ and $f^\star$ if and only if inequality~\eqref{int} is satisfied.


\smallskip We first establish the only if part. Summing the above two conditions leads to $-\langle g, x^\star-x \rangle \ge 2 (1-\tfrac{\mu}{L})^{-1} E(x,g,x^\star)$ which is equivalent to
\[(1-\tfrac{\mu}{L}) \langle g, x-x^\star \rangle  \ge \tfrac{1}{L}\|g \|^2+ \mu \|x-x^\star\|^2-\tfrac{2\mu}{L}\langle g, x-x^\star \rangle  \]
and can be rearranged as inequality~\eqref{int} in the Lemma.

\smallskip \noindent To prove the if part, we assume inequality~\eqref{int} holds and show that the above two conditions are satisfied with the following
\[ f_x = (1-\tfrac{\mu}{L})^{-1} E(x,g,x^\star), \quad f^\star=0 \]
It is easy to check that this choice of $f_x$ and $f^\star$ leads to an equality in the first condition. Furthermore, as the sum of the two inequalities was shown above to be equivalent to condition~\eqref{int}, we can conclude that the second inequality is also satisfied.
This concludes the proof of the equivalence.





Note that condition~\eqref{int} can be equivalently rewritten as  \begin{equation} \label{eq:geom}  \big\| g - \tfrac{L+\mu}{2} (x^\star-x) \big\| \le  \tfrac{L-\mu}{2} \big\|x^\star-x \big\|  \end{equation}
when $L<\infty$, which is easier to interpret geometrically: it states that $g$ must belong to a ball centered at $\tfrac{L+\mu}{2} (x^\star-x)$ with radius equal to  $\tfrac{L-\mu}{2} \big\|x^\star-x \big\|$ (to verify the equivalence, square both sides and distribute the left-hand side).
\end{proof}


\section{Minimizers of a sum of two convex functions}\label{Sec.1}

We first study the case of a sum of two smooth (strongly) convex functions. More precisely, given two points $x_1^\star, x^\star_2$, we want to characterize the set of potential minimizers of a sum $f_1(x)+f_2(x)$ where $f_1$ and $f_2$ are smooth and strongly convex and have a minimizer at $x_1^\star$ and $x_2^\star$ respectively.
 


\smallskip 
\begin{proposition}\label{Prop.1}
Let $0 \le \mu_1 < L_1 < \infty$ and $0 \le \mu_2 < L_2 < \infty$. Given two points $x_1^\star, x^\star_2 \in \mathbb{R}^n$, a point $x^\star \in \mathbb{R}^n$ is a minimizer of the sum $f_1(x)+f_2(x)$ for some functions $f_1\in\mathcal{M}_{\mu_1,L_1}(x^\star_1)$ and $f_2\in\mathcal{M}_{\mu_2,L_2}(x^\star_2)$ if and only if 
\begin{align}\label{M.inq}
\nonumber \left\| (L_1+\mu_1)(x^\star-x_1^\star)+(L_2+\mu_2)(x^\star-x_2^\star) \right\|\leq \\
(L_1-\mu_1)\| x^\star-x^\star_1 \|+(L_2-\mu_2)\| x^\star-x^\star_2 \|.
\end{align}
\end{proposition}

\smallskip 

\begin{proof}
Optimality conditions imply that $x^\star$ is a minimizer of the sum $f_1(x)+f_2(x)$ if and only if there exists subgradients of $f_1$ and $f_2$ at $x^\star$ that sum to zero, i.e.\@ if and only if there exists $g_1$ and $g_2$ such that
\[
g_1\in\partial f_1(x^\star), g_2\in\partial f_2(x^\star), g_1+g_2=0.
\]

Lemma \ref{Lem1} with the geometrical inequality \eqref{eq:geom} applied to each function $f_1$ and $f_2$ implies that $x^\star$ is a potential minimizer iff the following system has a solution,
\begin{align*}
& \big\| g_1-\tfrac{L_1+\mu_1}{2}(x^\star-x_1^\star) \big\|\leq
\tfrac{L_1-\mu_1}{2}\| x^\star-x^\star_1 \| ,\\
& \big\| g_2-\tfrac{L_2+\mu_2}{2}(x^\star-x_2^\star) \big\|\leq
\tfrac{L_2-\mu_2}{2}\| x^\star-x^\star_2 \|,\\
& g_1+g_2=0.
\end{align*}
By  replacing $g_2$ by $-g_1$, this system is equivalent to
\begin{align*}
& \big\| g_1-\tfrac{L_1+\mu_1}{2}(x^\star-x_1^\star) \big\|\leq
\tfrac{L_1-\mu_1}{2}\| x^\star-x^\star_1 \| ,\\
& \big\| g_1+\tfrac{L_2+\mu_2}{2}(x^\star-x_2^\star) \big\|\leq
\tfrac{L_2-\mu_2}{2}\| x^\star-x^\star_2 \|.
\end{align*}
Geometrically, this system is asking for $g_1$ to belong to the intersection of two balls. This is possible if and only if the distance between the centers of the two balls is less than or equal to the sum of their radii, which gives the desired inequality. 
\end{proof}

\begin{figure}[h]
    \centering
    \subfigure[$\mu_{1},\mu_2=1, L_1=5, L_2=15$]
    {
      \includegraphics[height=1.5in, width=1.5in]{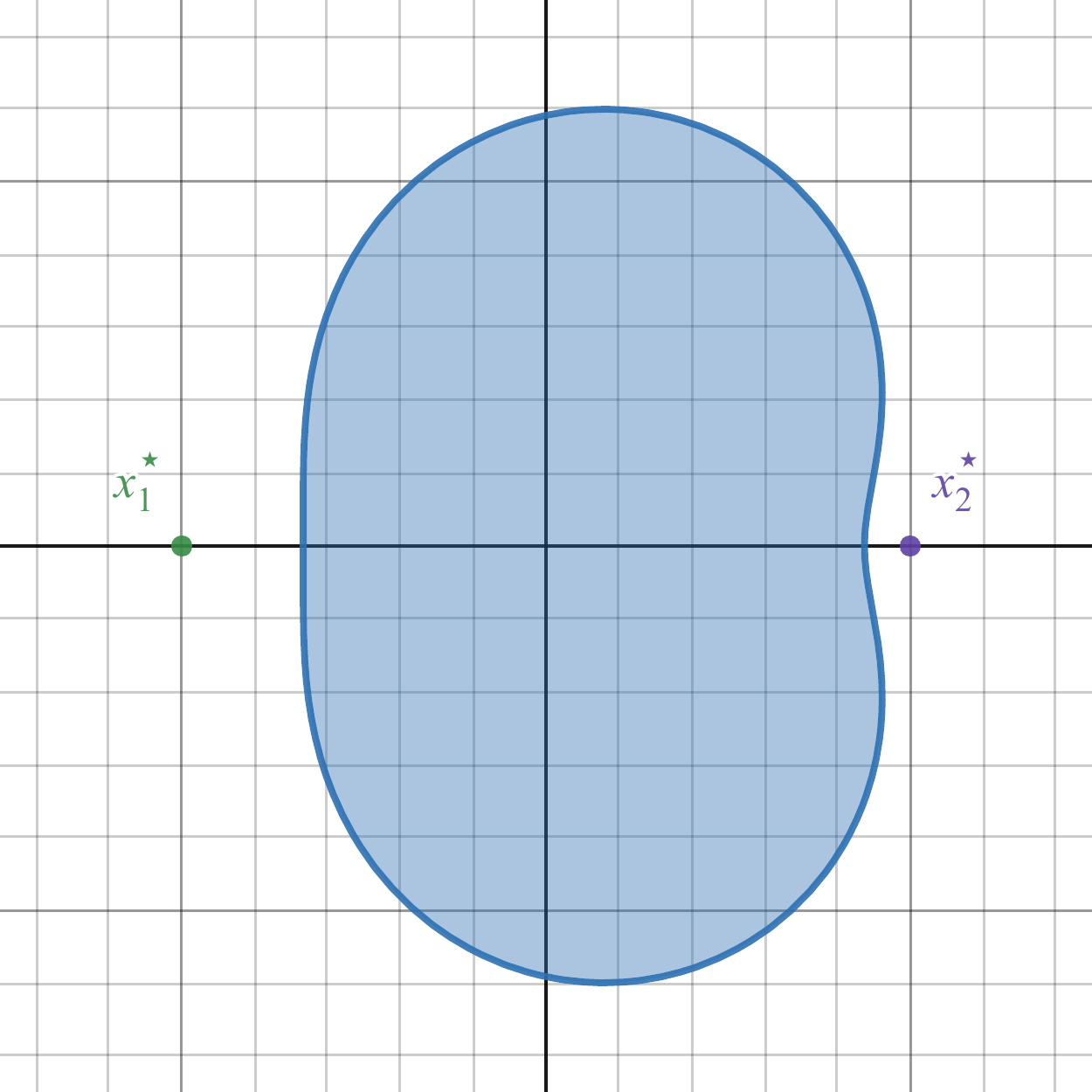}  
    }
    \subfigure[$\mu_1=1, \mu_2=4, L_{1},L_2=5 $]
    {
    \includegraphics[height=1.5in, width=1.5in]{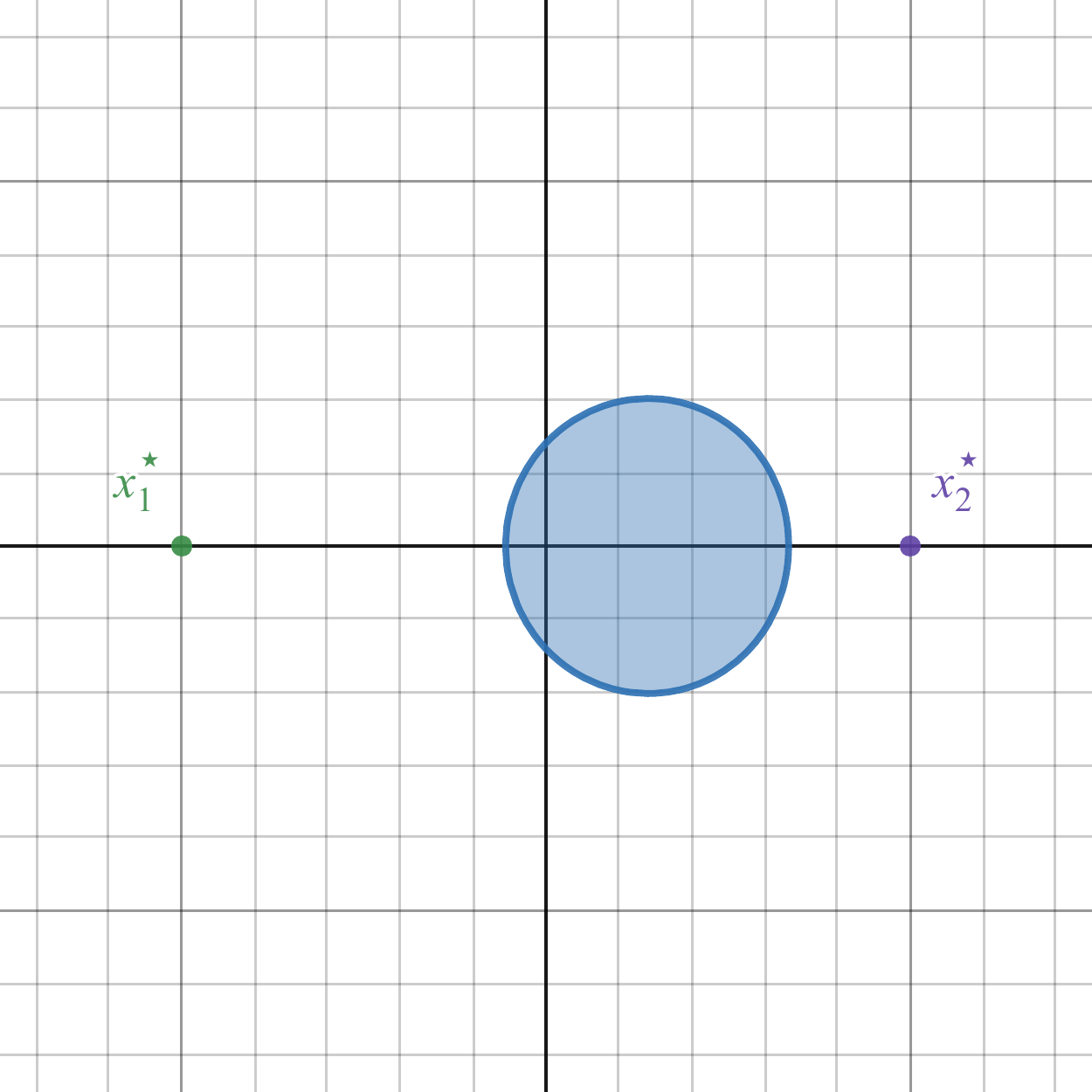}
    }
    \caption{Sets of potential minimizers $x^\star$  from Proposition \ref{Prop.1} 
    with $x_1^\star = (-1, 0)$ and $x_2^\star = (1, 0)$.}
    \label{fig.P1}
\end{figure}

\smallskip
One can verify that this set is always bounded provided $\mu_1>0$ or $\mu_2>0$. Indeed, as the norm of $x^\star$ tends to infinity, $x_1^\star$ and $x_2^\star$ become negligible in \eqref{M.inq}, leading in the limit to condition $(L_1+\mu_1+L_2+\mu_2)\|x^\star\| \leq (L_1-\mu_1+L_2-\mu_2)\| x^\star\|$ which cannot be satisfied unless $\mu_1+\mu_2=0$.
On the other hand, when $\mu_1=\mu_2=0$, condition \eqref{M.inq} becomes
\[ \left\| L_1(x^\star-x_1^\star)+L_2(x^\star-x_2^\star) \right\|\leq 
L_1 \| x^\star-x^\star_1 \|+L_2 \| x^\star-x^\star_2 \| \]
which is the triangle inequality. This implies that any point $x^\star \in \mathbb{R}^n$ can be a minimizer (which can also be derived from the choice $f_1(x)=f_2(x)=0$).

\smallskip Observe that the following point 
 \[ x^\star=\tfrac{L_1+\mu_1}{L_1+\mu_1+L_2+\mu_2}x^\star_1+\tfrac{L_2+\mu_2}{L_1+\mu_1+L_2+\mu_2}x^\star_2\]
satisfies inequality \eqref{M.inq} strictly when $x^\star_1\neq x^\star_2$, from which one can infer that the set of potential minimizers has a nonempty interior. 
Figure \ref{fig.P1} shows the set of potential minimizer $x^\star$ from Proposition \ref{Prop.1} for two different sets of parameters in $\mathbb{R}^2$.

\bigskip We now consider a situation where one of the functions is not smooth, assuming without loss of generality $L_2=\infty$.

\smallskip
\begin{proposition}\label{Prop.2}
Let $0 \le \mu_1 < L_1 < \infty$ and $0 \le \mu_2$. Given two points $x_1^\star, x^\star_2 \in \mathbb{R}^n$, a point $x^\star \in \mathbb{R}^n$ is a minimizer of the sum $f_1(x)+f_2(x)$ for some functions $f_1\in\mathcal{M}_{\mu,L}(x^\star_1)$ and $f_2\in\mathcal{M}_{\mu,L}(x^\star_2)$  if and only if the following holds
\begin{align}\label{M.inq2}
\nonumber  \mu_2\|x^\star-x_2^\star\|^2+\tfrac{L_1+\mu_1}{2}\langle x^\star-x_1^\star, x^\star-x^\star_2\rangle \leq \\
\tfrac{L_1-\mu_1}{2}\| x^\star-x^\star_1 \|\| x^\star-x^\star_2 \|.
\end{align}
\end{proposition}

\smallskip \begin{proof}
Due to Lemma \ref{Lem1} and the optimality conditions for the sum $f_1(x)+f_2(x)$, $x^\star$ is a potential minimizer iff the following system has a solution (using both \eqref{int} and \eqref{eq:geom})
\begin{align*}
& \big\| g_1-\tfrac{L_1+\mu_1}{2}(x^\star-x_1^\star) \big\|\leq
\tfrac{L_1-\mu_1}{2}\| x^\star-x^\star_1 \| ,\\
& \langle g_2, x^\star-x_2^\star\rangle \geq
\mu_2\| x^\star-x^\star_2 \|^2,\\ 
& g_1+g_2=0.
\end{align*}
Replacing $g_2$ by $-g_1$, the above system becomes
\begin{align*}
& \big\| g_1-\tfrac{L_1+\mu_1}{2}(x^\star-x_1^\star) \big\|\leq
\tfrac{L_1-\mu_1}{2}\| x^\star-x^\star_1 \|,\\
& \langle -g_1, x^\star-x_2^\star\rangle \geq
\mu_2\| x^\star-x^\star_2 \|^2.
\end{align*}
The desired inequality follows from the necessary and sufficient condition for the intersection between a ball and a half-space to be nonempty. 
\end{proof}

Proposition \ref{Prop.2} also corresponds to Proposition \ref{Prop.1} in the limit when $L_2$ tends to infinity; see the discussion before Corollary \ref{Coro.3_2}. 
Figure \ref{fig.P2} shows the set of potential minimizer $x^\star$ from Proposition \ref{Prop.2} for two different sets of parameters in $\mathbb{R}^2$. 
As before, this set is bounded when $\mu_1+\mu_2>0$.
 
\begin{figure}
    \centering
    \subfigure[$\mu_1=1, \mu_2=3,L_1=4$.]
    {
      \includegraphics[height=1.5in, width=1.5in]{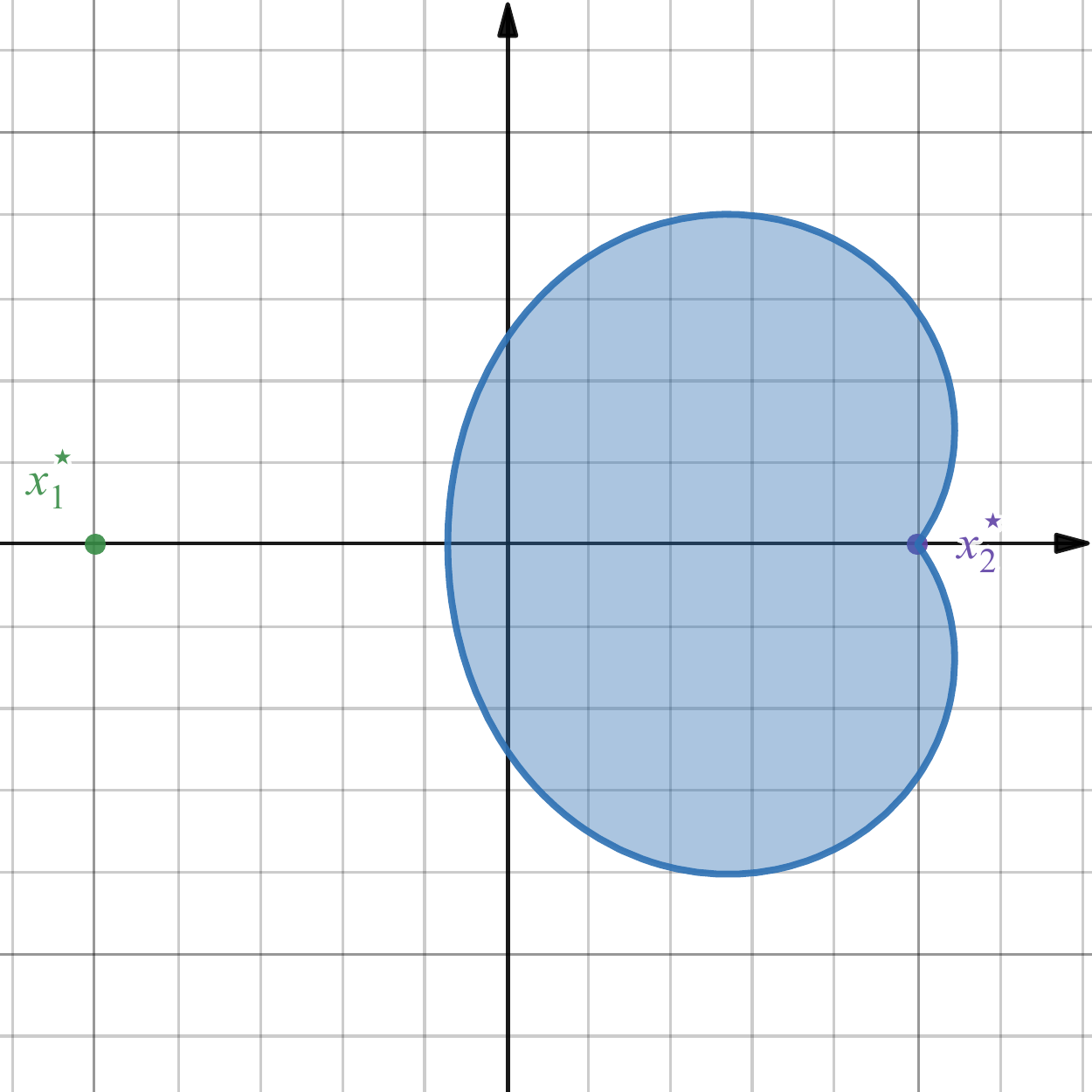}  
    }
    \subfigure[$\mu_1=1, \mu_2=0,L_1=1.5$.]
    {
    \includegraphics[height=1.5in, width=1.5in]{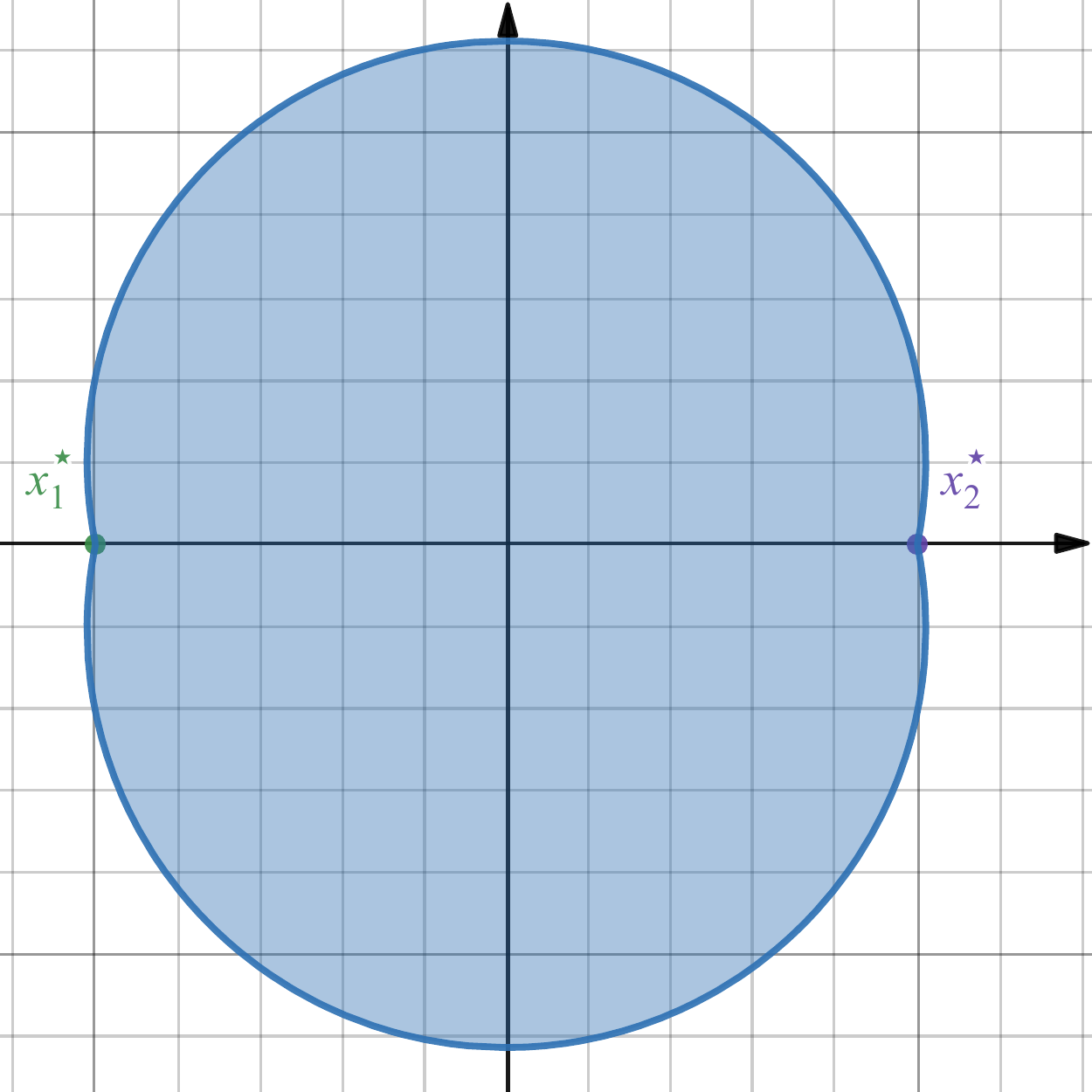}
    }
    \caption{Sets of potential minimizers $x^\star$  from Proposition \ref{Prop.2} 
    with $x_1^\star = (-1, 0)$ and $x_2^\star = (1, 0)$.}
    \label{fig.P2}
\end{figure}

\bigskip 
When none of the functions is smooth, i.e. when $L_1=L_2=\infty$, the situation becomes qualitatively different. Indeed, as shown in \cite{kuwaranancharoen2023minimizer}, the set of potential minimizers becomes unbounded and its closure is equal to the whole $\mathbb{R}^n$ when $n\geq 2$. Hence, similarly to \cite{kuwaranancharoen2023minimizer}, we consider an additional condition on the minimizer: we assume that the subgradients of $f_1$ and $f_2$ at the minimizer $x^\star$ have their norm bounded by a constant $B$ (note that since those subgradients sum to zero, they have the same norm).





\smallskip Since the case $\mu_1=\mu_2=0$ is uninteresting (all points can be minimizers in that case), we assume from here that $\mu_1 + \mu_2 > 0$, i.e. that function $f$ is strongly convex.
Using condition~\eqref{int} from Lemma~\ref{Lem1}, we have that the subgradient $g_1$ at a potential minimizer $x^\star$ must satisfy $\|g_1\|\geq \mu_1 \|x_1^\star-x^\star\|$, hence the upper bound $B$ on its norm must satisfy $B \ge \mu_1 \|x_1^\star-x^\star\|$. Similarly, we must have $B \ge \mu_2 \|x_2^\star-x^\star\|$ (also showing that the set of potential minimizers is bounded). Combining with $\|x_1^\star-x^\star\|+ \|x_2^\star-x^\star\| \ge\|x_1^\star-x_2^\star\|$ (triangle inequality), these two inequalities imply that 
\begin{equation} \label{boundB} B\geq \tfrac{\mu_1\mu_2}{\mu_1+\mu_2}\|x_1^\star-x_2^\star\| .\end{equation}

Define the focal point $x_f=\tfrac{\mu_1}{\mu_1+\mu_2}x^\star_1+\tfrac{\mu_1}{\mu_1+\mu_2}x^\star_2$, easily seen to be the minimizer of the sum of $f_1(x)=\tfrac{\mu_1}{2}\|x-x_1^\star\|^2$ and $f_2(x)=\tfrac{\mu_2}{2}\|x-x_2^\star\|^2$. Since we also have $\|\nabla f_1(x_f)\| = \tfrac{\mu_1\mu_2}{\mu_1+\mu_2}\|x_1^\star-x_2^\star\|$, the bound \eqref{boundB} is necessary and sufficient for the set of minimizers to be nonempty.





\smallskip We believe the following characterization is easier to state and slightly more complete than the one proposed in \cite{kuwaranancharoen2023minimizer}.

\smallskip \begin{proposition}\label{Prop.3}
Let $\mu_1, \mu_2 \ge 0$ such that $\mu_1+\mu_2>0$. Given two points $x_1^\star\neq x^\star_2 \in \mathbb{R}^n$ and a bound $B$ satisfying \eqref{boundB}, a point $x^\star \in \mathbb{R}^n$ is a minimizer of the sum $f_1(x)+f_2(x)$ for some functions $f_1\in\mathcal{M}_{\mu_1,\infty}(x^\star_1)$, $f_2\in\mathcal{M}_{\mu_2,\infty}(x^\star_2)$ with $\|g\| \le B$ for some $g\in \partial f_1(x^\star)$
if and only if 
it satisfies 
\[ \mu_1\|x^\star-x^\star_1\|\leq B \text{ and } \mu_2\|x^\star-x^\star_2\|\leq B \] 
and at least an additional inequality among the following:
\begin{align*}
&(i)\ 
  \mu_1 \langle x_1^\star-x^\star, x^\star-x^\star_2 \rangle \geq \mu_2 \| x^\star-x^\star_2 \|^2\\
 & (ii)\  
    \mu_2 \langle x_2^\star-x^\star, x^\star-x^\star_1 \rangle \geq \mu_1 \| x^\star-x^\star_1 \|^2\\
  & (iii)\
    \det\left(M(x^\star, x^\star_1, x^\star_2, B^2)\right)\geq 0.
\end{align*}
\end{proposition}

where matrix $M(x^\star, x^\star_1, x^\star_2, \alpha)$ be given as follows
$$
\begin{pmatrix}
    1 & \tfrac{\langle x^\star-x_1^\star, x^\star-x^\star_2 \rangle}{\|x^\star-x^\star_1\| \|x^\star-x^\star_2\|} & \mu_1\|x^\star-x^\star_1\| 
    \\
  \tfrac{\langle x^\star-x_1^\star, x^\star-x^\star_2 \rangle}{\|x^\star-x^\star_1\| \|x^\star-x^\star_2\|} & 1 & -\mu_2\|x^\star-x^\star_2\|   
   \\
    \mu_1\|x^\star-x^\star_1\| & -\mu_2\|x^\star-x^\star_2\| & \alpha
\end{pmatrix}.
$$

\smallskip \begin{proof}
Using Lemma \ref{Lem1}, we see that $x^\star$ is a potential minimizer if and only if the optimal value of the following convex problem is less than or equal to $B^2$,
\begin{align}\label{CQP}
 \min\ \|g\|^2\ \st\ &\ \langle g, x^\star_1-x^\star \rangle \leq -\mu_1 \| x^\star-x^\star_1 \|^2 \\
\nonumber  & \ \langle g, x^\star-x^\star_2 \rangle \leq -\mu_2 \| x^\star-x^\star_2 \|^2
\end{align}
   It is seen that the problem has a unique solution. Due to the optimality condition, the feasible point $g$ is the optimal solution of problem \eqref{CQP} iff there exist $\lambda_1, \lambda_2\geq 0$ with 
\begin{align*}
   &  g=\lambda_1(x^\star-x^\star_1)+\lambda_2(x_2^\star-x^\star), \\
   & \lambda_1\left(\langle g, x^\star_1-x^\star \rangle+\mu_1 \| x^\star-x^\star_1 \|^2\right)=0\\
 &  \lambda_2\left(\langle g, x^\star-x^\star_2 \rangle+\mu_2 \| x^\star-x^\star_2 \|^2\right)=0.
\end{align*}
  As $x_1^\star\neq x^\star_2$, the optimal value cannot be zero, and at least one constraint is active at the optimal solution. 
    If the first constraint is active, the optimal value is less that or equal to $B^2$ iff 
  $$
  \mu_1\|x^\star-x^\star_1\|\leq B,
  \mu_1 \langle x_1^\star-x^\star, x^\star-x^\star_2 \rangle \geq \mu_2 \| x^\star-x^\star_2 \|^2.
  $$
  The case that the second is active is shown analogously. Now, we consider the case that both constraints are active. Without loss of generality, we assume that $x^\star-x^\star_1$ and $x^\star-x^\star_2$ are not parallel. The case that these vectors are parallel is covered by the former cases. Consider the Gram matrix composed of  $x^\star-x^\star_1$, $x^\star-x^\star_2$ and $g$. In this case, problem \eqref{CQP} and the following semi-definite program share the same optimal value,
  \begin{equation} \label{SDP} \min \ \alpha \  \st  \ M(x^\star, x^\star_1, x^\star_2, \alpha)\succeq 0.
\end{equation}
Note that we factor $\|x^\star-x^\star_1\|^2$ and $\|x^\star-x^\star_2\|^2$ in problem \eqref{SDP} as $x^\star\neq x_1^\star, x_2^\star$.  The optimal value of semi-definite program  \eqref{SDP} is less than or equal to $B^2$ iff 
$ \det\left(M(x^\star, x^\star_1, x^\star_2, B^2)\right)\geq 0$. This follows from the fact that a symmetric matrix is positive definite iff its leading principal minors are positive. To complete the proof, we need to show that if  $x^\star-x^\star_1$ and $x^\star-x^\star_2$ are parallel and $\det(M(x^\star, x^\star_1, x^\star_2, B^2)\geq 0$, then $x^\star$ is a potential minimizer. In this case, we have 
$$
\det\left(M(x^\star, x^\star_1, x^\star_2, \alpha)\right)=-\|\mu_1(x^\star-x^\star_1)+\mu_2(x^\star-x^\star_2)\|^2,
$$
which implies $x^\star=x_f$ and the proof is complete. 
\end{proof}

\smallskip Closedness of the potential set of minimizers under the assumptions of Proposition \ref{Prop.3} follows from closedness of the intersection of the feasible region of \eqref{CQP} and $\|g\|\leq B$. 
Figure \ref{fig.P3} shows the set of potential minimizer $x^\star$ from Proposition \ref{Prop.3} for two different sets of parameters in $\mathbb{R}^2$ (colors green and red correspond to the first two inequalities, color blue corresponds to the third determinant condition). Note that a potential minimizer may satisfy more than one system among the three systems given in Proposition \ref{Prop.3}.

\begin{figure}
    \centering
    \subfigure[$\mu_1=1.75,\mu_2=2, B=3$.]
    {
      \includegraphics[height=1.5in, width=1.5in]{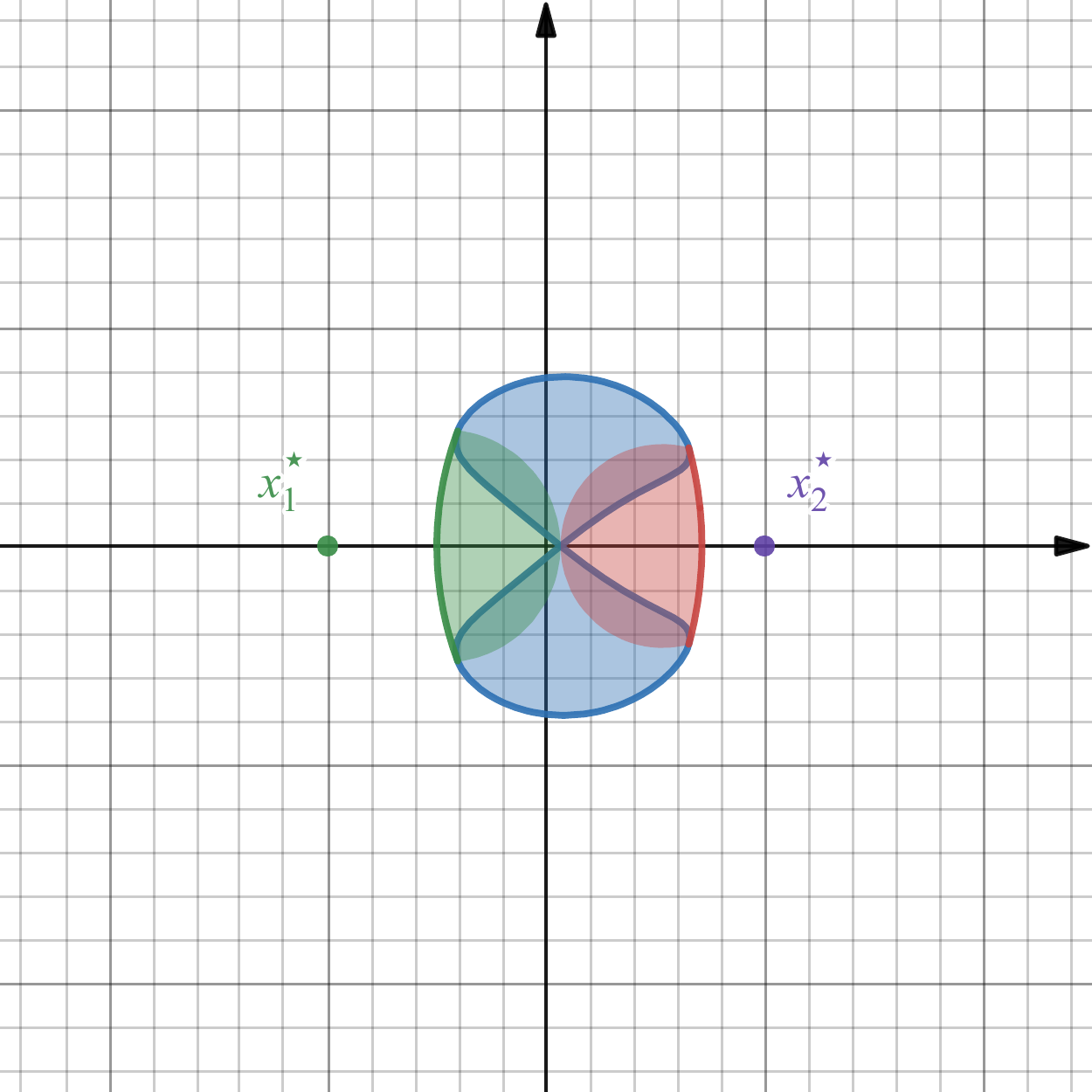}  
    }
    \subfigure[$\mu_1=1.75,\mu_2=0, B=3$.]
    {
    \includegraphics[height=1.5in, width=1.5in]{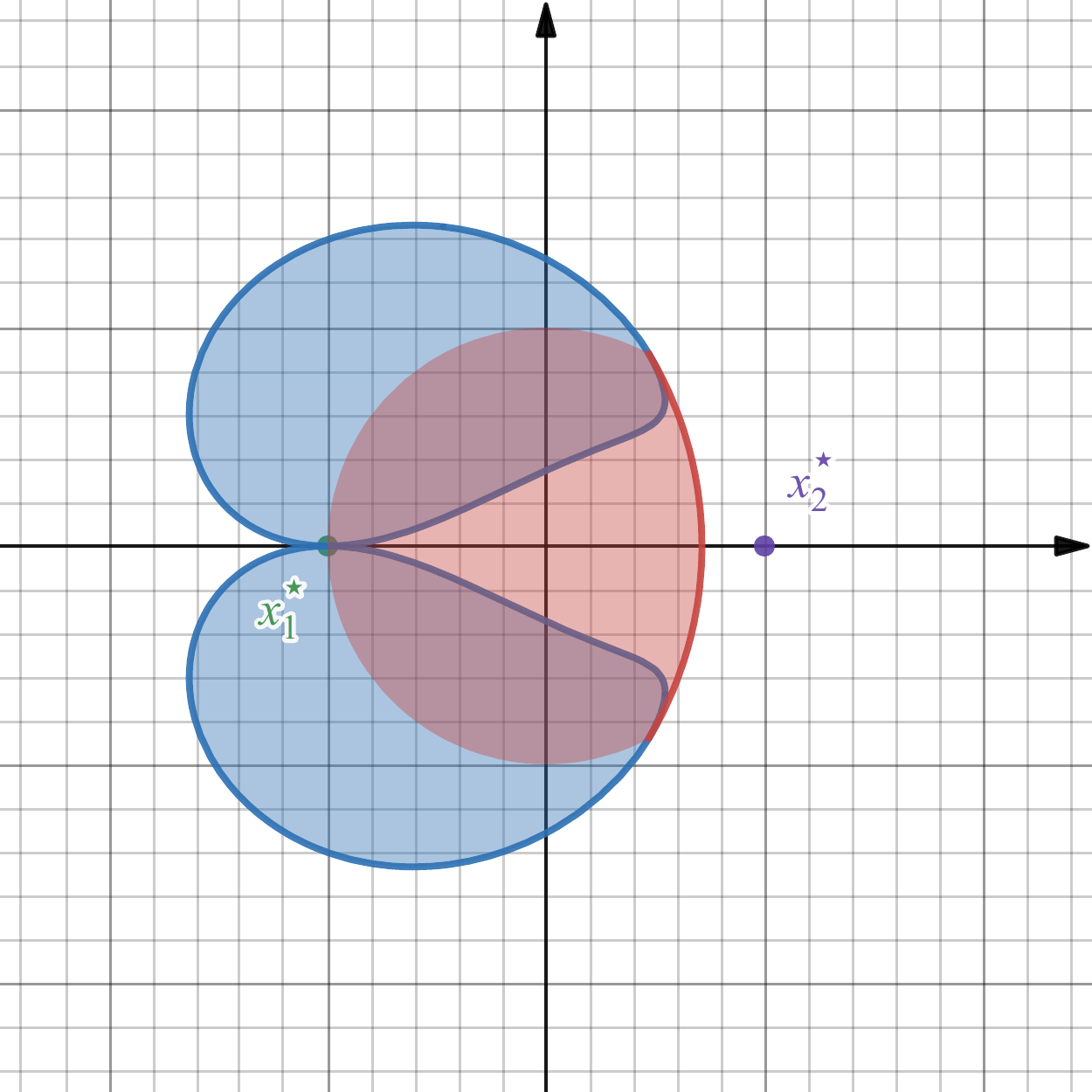}
    }
    \caption{Sets of potential minimizers $x^\star$  from Proposition \ref{Prop.3} 
    with $x_1^\star = (-1, 0)$ and $x_2^\star = (1, 0)$.}
    \label{fig.P3}
\end{figure}

\bigskip Our last theorem bounds the distance between any minimizer in Proposition~\ref{Prop.3} and the focal point $x_f$. 

\smallskip\begin{theorem}\label{T.1}
Any minimizer $x^*$ verifying the assumptions of Proposition~\ref{Prop.3} 
must satisfy
\begin{align*} 
& \nonumber \left\|x^\star-x_f\right\|^2\leq \min\big(\tfrac{B}{\mu_1+\mu_2} \left\|x_1^\star-x_2^\star\right\|-\tfrac{\mu_1\mu_2}{(\mu_1+\mu_2)^2} \\
& \hspace{3cm}  \left\|x_1^\star-x_2^\star\right\|^2, \tfrac{B^2}{(\mu_1+ \mu_2)^2} \big).
\end{align*} 
\end{theorem}

\smallskip\begin{proof} 
We may assume w.l.o.g. that $x^\star_1=\tfrac{\mu_2\left\|x_1^\star-x_2^\star\right\|}{\mu_1+\mu_2}e_1$ and $x^\star_2=-\tfrac{\mu_1\left\|x_1^\star-x_2^\star\right\|}{\mu_1+\mu_2}e_1$, 
which implies that $x_f=0$ and  
\begin{align}\label{Pr1_1}
\mu_1\left\|x_1^\star\right\|^2+\mu_2\left\|x_2^\star\right\|^2=
 \tfrac{\mu_1\mu_2}{\mu_1+\mu_2}\left\|x_1^\star-x_2^\star\right\|^2.
\end{align}

\noindent First, we establish 
$$
\left\|x^\star-x_f\right\|^2\leq \tfrac{B}{\mu_1+\mu_2} \left\|x_1^\star-x_2^\star\right\|-\tfrac{\mu_1\mu_2}{(\mu_1+\mu_2)^2} \left\|x_1^\star-x_2^\star\right\|^2.
$$ 
By \eqref{int}, we have 
\begin{align*}
\langle g, x^\star-x^\star_1\rangle \geq \mu_1 \left\|x^\star-x^\star_1\right\|^2,\\
\langle -g, x^\star-x^\star_2\rangle \geq \mu_2 \left\|x^\star-x^\star_2\right\|^2.
\end{align*}
By summing these inequalities, we get
\begin{align*}
\langle g, x^\star_2-x^\star_1\rangle \geq \mu_1 \left\|x_1^\star\right\|^2+\mu_2 \left\|x_2^\star\right\|^2
+(\mu_1+\mu_2)\left\|x^\star\right\|^2.
\end{align*}
and we conclude with Cauchy–Schwarz inequality and \eqref{Pr1_1}
$$
(\mu_1+\mu_2)\left\|x^\star\right\|^2\leq 
{B} \left\|x_1^\star-x_2^\star\right\|-\tfrac{\mu_1\mu_2}{\mu_1+\mu_2} \left\|x_1^\star-x_2^\star\right\|^2.
$$

Now, we prove that $\|x^\star-x_f\|^2\leq  \tfrac{B^2}{(\mu_1+ \mu_2)^2}$.
Suppose that $\gamma= \tfrac{1}{(\mu_1+ \mu_2)^2}$. Consider the following inequalities,
\begin{align*}
& \gamma\left( B^2-\|g\|^2 \right)\geq 0,\\
& 2\mu_1\gamma\left( \langle g, x^\star-x^\star_1\rangle - \mu_1 \left\|x^\star-x^\star_1\right\|^2 \right)\geq 0,\\
& 2\mu_2\gamma\left( \langle -g, x^\star-x^\star_2\rangle - \mu_2 \left\|x^\star-x^\star_2\right\|^2 \right)\geq 0.
\end{align*}
By summing these inequalities, we get
$$
\gamma B^2-\left\|x^\star\right\|^2-\gamma\left\| \mu_1 x_1^\star-\mu_2 x_2^\star+(\mu_2-\mu_1)x^\star+g \right\|^2\geq 0,
$$
which implies the desired inequality.
\end{proof}
 
\section{Minimizers of a sum of $m$ convex functions}\label{Sec.2}
We now extend the results of Section \ref{Sec.1} to sums of more than two functions,  
and start with the smooth case.

\begin{theorem}\label{Theo.1}
Let $0 \le \mu_i < L_i < \infty$ for all $i\in [m]$. Given $m$ points $x_i^\star \in \mathbb{R}^n$, $i\in [m]$, a point $x^\star \in \mathbb{R}^n$ is a minimizer of the sum $\sum_{i=1}^m f_i(x)$ for some functions $f_i\in\mathcal{M}_{\mu_i,L_i}(x^\star_i)$, $i\in [m]$ if and only if
\begin{equation}\label{Th1.inq}
\big\| \sum_{i\in [m]} (L_i+\mu_i)(x^\star-x_i^\star) \big\|\leq 
\sum_{i\in [m]} (L_i-\mu_i)\| x^\star-x^\star_i \|.
\end{equation}
\end{theorem}

\smallskip\begin{proof}
The proof is analogous to that of Proposition  \ref{Prop.1}. Due to the optimality conditions, $x^\star$ is a minimizer of $\sum_{i\in [m]}f_i(x)$ iff there exist $g_i\in\partial f_i(x^\star)$, $i\in [m]$ such that $\sum_{i\in[m]} g_i=0$. Hence, by Lemma \ref{Lem1}, $x^\star\in\argmin \sum_{i\in [m]}f_i$  with $f_i\in\mathcal{M}_{\mu_i, L_i}(x^\star_i)$ for $i\in [m]$ iff there exists $g_i, i \in [m]$ satisfying $\sum_{i\in[m]} g_i=0$ and \[
 \big\| g_i-\tfrac{L_i+\mu_i}{2}(x^\star-x_i^\star) \big\|\leq
\tfrac{L_i-\mu_i}{2}\| x^\star-x^\star_i \|,\ i\in [m]\]
Below we use $L^+_i = \frac{L_i+\mu_i}{2}$ and $L^-_i = \frac{L_i-\mu_i}{2}$ for convenience.
Replacing $g_m$ with $-\sum_{i\in[m-1]} g_i$, we obtain
\begin{align*}
& \big\| g_i- L^+_i (x^\star-x_i^\star) \big\|\leq
L^-_i\| x^\star-x^\star_1 \|,\ i\in [m-1]\\
& \big\| \textstyle\sum_{i\in [m-1]} g_i+L^+_m(x^\star-x_m^\star) \big\|\leq
L^-_m\| x^\star-x^\star_m \|.
\end{align*}
We check the consistency of this system by an elimination method. We first investigate the solvability with respect to $g_{m-1}$ while $g_1, \dots, g_{m-2}$ are fixed. As $g_{m-1}$ only appears in the following two inequalities, defining two balls
\begin{align*}
& \big\| g_{m-1}-L^+_{m-1}(x^\star-x_{m-1}^\star) \big\|\leq
L^-_{m-1}\| x^\star-x^\star_{m-1} \|\\
& \big\| \textstyle g_{m-1} + \sum_{i\in [m-2]} g_i+L^+_{m}(x^\star-x_m^\star) \big\|\leq
L^-_{m}\| x^\star-x^\star_m \|,
\end{align*}
the system has a solution in terms of $g_{m-1}$ iff distance between the two ball centers is not greater than the sum of the two radii, leading to the inequality
\begin{align*}
& \big\| \textstyle\sum_{i\in [m-2]} g_i+L^+_m(x^\star-x_m^\star)+L^+_{m-1}(x^\star-x_{m-1}^\star) \big\|\\
\leq &\ L^-_m \| x^\star-x^\star_m \|+L^-_{m-1}\| x^\star-x^\star_{m-1} \|.
\end{align*}
Consider now at $g_{m-2}$, which must again lie in the intersection of two balls: applying the same principle leads to an inequality of the same form as above, with one more term with $L^+_{m-2}$ inside of the left-hand side norm, and one more norm term with $L^-_{m-2}$ on the right-hand side.
A simple induction argument then completes the proof.
\end{proof}

\bigskip The potential minimizer set in Theorem \ref{Theo.1} is closed and bounded set if $\sum_{i\in [m]} \mu_i>0$. If $x^\star_i$'s are distinct, inequality \eqref{Th1.inq} is satisfied strictly for the focal point $x_f=\tfrac{1}{\sum_{i\in [m]} L_i+\mu_i}\sum_{i\in [m]} (L_i+\mu_i) x^\star_i$, and the interior of the minimizer set is non-empty.  

\bigskip We now consider the case where (exactly) one summand  is not smooth, using a limiting argument. Rewriting \eqref{Th1.inq} as 
\begin{align}\label{Cor1.inq}
\textstyle \nonumber & \textstyle \frac12\sum_{i\in [m]} \mu_i L_i\| x^\star-x_i^\star\|^2 + \sum_{\substack{i<j}} L^+_i L^+_j \langle x^\star-x_i^\star,  x^\star-x_j^\star \rangle
\\
 & \textstyle \leq \textstyle \sum_{\substack{i<j, i, j\in [m]}} L^-_i L^-_j \| x^\star-x^\star_i\| \| x^\star-x^\star_j\|
\end{align}
dividing both sides by $L_m$ and taking the limit as $L_m\to\infty$, we obtain the inequality of the following Corollary.

\smallskip \begin{corollary}\label{Coro.3_2}
Let $0 \le \mu_i < L_i < \infty$ for all $i\in [m-1]$ and $0 \le \mu_m < L_m = \infty$. Given $m$ points $x_i^\star \in \mathbb{R}^n$, a point $x^\star \in \mathbb{R}^n$ is a minimizer of the sum $\sum_{i=1}^m f_i(x)$ for some functions $f_i\in\mathcal{M}_{\mu_i,L_i}(x^\star_i)$, $i\in [m]$ if and only if
\begin{align*}
 &  \sum_{i\in [m-1]} \tfrac{L_i+\mu_i}{2} \langle x^\star-x_i^\star, x^\star-x_m^\star \rangle + \mu_m\| x^\star-x_m^\star\|^2  \\
\leq  & \sum_{i\in [m-1]} \tfrac{L_i-\mu_i}{2}\| x^\star-x^\star_i\| \| x^\star-x^\star_m\|.
\end{align*}
\end{corollary}

\smallskip Observe that if more than one summand is not smooth, the set of potential minimizers may become very large set, as its closure may be whole space \cite{kuwaranancharoen2023minimizer}.

\bigskip In some cases, some functions are only unknown. The following proposition addresses this situation.  For the simplicity of the presentation, we assume that the known functions $f_i$, $i\in [\ell]$, are differentiable. 

\begin{proposition}\label{Por.P1}
Let $f_i\in\mathcal{F}_{0, \infty}$, $i\in [\ell]$ be $l$ known differentiable convex functions. Let $0 \le \mu_i < L_i < \infty$ for all $i\in [\ell,m]$. Given $m-l$ points $x_i^\star \in \mathbb{R}^n$, $i\in [\ell+1,m]$, a point $x^\star \in \mathbb{R}^n$ is a minimizer of the sum $\sum_{i=1}^m f_i(x)$ for some functions $f_i\in\mathcal{M}_{\mu_i,L_i}(x^\star_i)$, $i\in [\ell+1,m]$ if and only if
\begin{align*}
& \big\|  2\sum_{i\in [\ell]} \nabla f_i(x^\star) +\sum_{i\in [\ell+1, m]} (L_i+\mu_i)(x^\star-x_i^\star) \big\|\leq \\
& \ \ \ \sum_{i\in [\ell+1, m]} (L_i-\mu_i)\| x^\star-x^\star_i \|.
\end{align*}
\end{proposition}

\smallskip\begin{proof}
By Lemma \ref{Lem1}, $x^\star$ a potential minimizer iff the following system has a solution,
\begin{align*}
& \big\| g_i-\tfrac{L_i+\mu_i}{2}(x^\star-x_i^\star) \big\|\leq
\tfrac{L_i-\mu_i}{2}\| x^\star-x^\star_1 \|, i\in [\ell+1, m]\\
& \big\| \sum_{i\in [\ell+1, m-1]} g_i+ \sum_{i\in [\ell]} \nabla f_i(x^\star)+\tfrac{L_m+\mu_m}{2}(x^\star-x_m^\star) \big\|\leq\\
& \ \tfrac{L_m-\mu_m}{2}\| x^\star-x^\star_m \|.
\end{align*}
Analogous to the proof of Theorem \ref{Theo.1}, one can infer the desired inequality and the proof is complete.  
\end{proof}

Similar to Corollary \ref{Cor1.inq}, we can characterize the potential set when $L_m=\infty$ by using Proposition \ref{Por.P1}. The following corollary states this point.

\smallskip\begin{corollary}
Let $f_i\in\mathcal{F}_{0, \infty}$, $i\in [\ell]$, be known and differentiable convex functions. Assume that $L_i<\infty$ for $i\in [\ell+1, m-1]$ and $\mu_i$ and $x^\star_i$ for $i\in [\ell+1, m]$ are given. The point $x^\star\in\argmin\sum_{i\in [m]} f_i(x)$ for some
 $f_i\in\mathcal{M}_{\mu_i, L_i}(x^\star_i)$,  $i\in [\ell+1, m-1]$,
and $f_m\in\mathcal{M}_{\mu_m, \infty}(x_m^\star)$  iff 
\begin{align*}
 &  \mu_m\| x^\star-x_m^\star\|^2+ \sum_{i\in[\ell]} \langle \nabla f(x^\star), x^\star-x_m^\star \rangle+\\
& \sum_{i\in[\ell+1, m-1]} \tfrac{L_i+\mu_i}{2} \langle x^\star-x_i^\star, x^\star-x_m^\star \rangle \leq \\
& \ \ \ \sum_{i\in[\ell+1, m-1]} \tfrac{L_i-\mu_i}{2}\| x^\star-x^\star_i\| \| x^\star-x^\star_m\|.
\end{align*}
\end{corollary}
\bigskip

\section{Bounds when summand minimizers lie in a ball}\label{Sec.J}
Hendrickx et al. \cite{hendrickx2020stability} consider a sum of smooth strongly convex functions $f_i\in\mathcal{F}_{\mu, L}(\mathbb{R}^n)$ in the case that all summand minimizers $x_i^\star$ lie in a unit ball centered at the origin, i.e. $\|x^\star_i\|\leq 1$ for all $i \in [m]$. In \cite[Theorem 1]{hendrickx2020stability} they prove the following bound on the norm of a potential minimizer $x^\star$
\[ \|x^\star\|\leq 1+\sqrt{\kappa}, \quad \text{where $\kappa=\tfrac{L}{\mu}$.}\]
\noindent The following theorem gives a tighter bound. 

\smallskip\begin{theorem}\label{Theo.2}
Let $f_i\in\mathcal{M}_{\mu, L}(x^\star_i)$  with $0<\mu< L<\infty$, $i\in [m]$, and $\|x^\star_i-\bar x\|\leq 1$.
 If $x^\star\in\argmin \sum_{i=1}^m f_i(x)$  then 
$$
\|x^\star-\bar x\|\leq \tfrac{1}{2}\left(\sqrt{\kappa}+\tfrac{1}{\sqrt{\kappa}}\right), \quad \text{where $\kappa=\tfrac{L}{\mu}$.}
$$
\end{theorem}
\smallskip \begin{proof}
Without loss of generality, we may assume that $\bar x=0$. Since $2\| x^\star-x^\star_i\| \| x^\star-x^\star_j\|\leq \| x^\star-x^\star_i\|^2+ \| x^\star-x^\star_j\|^2$, inequality \eqref{Cor1.inq} may be written as
\begin{align}\label{T2.1}
\nonumber & 4\kappa\sum_{i\in [m]} \| x^\star-x_i^\star\|^2+2(\kappa+1)^2
\sum_{\substack{i<j \\ i, j\in [m]}}  \langle x^\star-x_i^\star, x^\star-x_j^\star \rangle \\ 
& -(\kappa-1)^2\sum_{\substack{i<j \\ i, j\in [m]}} \left( \| x^\star-x^\star_i\|^2+ \| x^\star-x^\star_j\|^2 \right)=\\
\nonumber &  \kappa\left((\sqrt{\kappa}+\tfrac{1}{\sqrt{\kappa}})^2-m(\sqrt{\kappa}-\tfrac{1}{\sqrt{\kappa}})^2\right)\sum_{i\in [m]} \| x^\star-x_i^\star\|^2+\\
\nonumber & 
2\kappa\left( \sqrt{\kappa}+\tfrac{1}{\sqrt{\kappa}} \right)^2\sum_{\substack{i<j \\ i, j\in [m]}}  \langle x^\star-x_i^\star, x^\star-x_j^\star \rangle\leq 0.
\end{align}
By the assumptions, we have for $i\in\{1, \dots, m\}$
\begin{align}\label{T2.2}
\|x^\star-x^\star_i\|^2+\|x^\star\|^2+2\langle x^\star-x^\star_i, x^\star\rangle-1\leq 0.
\end{align}
Suppose that $\alpha=\big(\tfrac{1}{2m\sqrt{\kappa}(\sqrt{\kappa}-\tfrac{1}{\sqrt{\kappa}})}\big(\sqrt{\kappa}+\tfrac{1}{\sqrt{\kappa}}\big)\big)^2$. By multiplying inequalities \eqref{T2.1} and \eqref{T2.2} by $\alpha$ and $\tfrac{1}{4m}\big(\sqrt{\kappa}+\tfrac{1}{\sqrt{\kappa}}\big)^2$, respectively, and summing these $m+1$ inequalities we get
\begin{align*}
& \|x^\star\|^2-\tfrac{1}{4}\big(\sqrt{\kappa}+\tfrac{1}{\sqrt{\kappa}}\big)^2+\big\| \tfrac{1}{2} (\sqrt{\kappa}-\tfrac{1}{\sqrt{\kappa}} )x^\star+\sqrt{\alpha}(\kappa+1) \\
& \ \ \sum_{i\in[m]} (x^\star-x^\star_i)  \big\|^2\leq 0.
\end{align*}
The above inequality implies the desired bound and the proof is complete. 
\end{proof}

We now addresses the setting where one of the functions (w.l.o.g. $f_m$) is an arbitrary closed proper convex function.

\smallskip\begin{theorem}\label{Theo.3}
Let $f_i\in\mathcal{M}_{\mu, L}(x_i^\star)$, $i\in[m-1]$,  with $0<\mu< L<\infty$ and let
$f_m\in\mathcal{M}_{0, \infty}(x_m^\star)$. Assume that 
 $\|x^\star_i-\bar x\|\leq 1$ for $i\in[m]$. 
 If  $x^\star\in\argmin \sum_{i=1}^m f_i(x)$ then 
$$
\|x^\star-\bar x\|\leq \sqrt{\kappa+1}, \ \   \text{where $\kappa=\tfrac{L}{\mu}$.}
$$
\end{theorem}
\begin{proof}
The proof is analogous to that of Theorem \ref{Theo.2}. We assume that $\bar x=0$. By Corollary \ref{Coro.3_2}, we have
\begin{align}\label{T3.1}
\nonumber & (\kappa+1)\sum_{i\in[m-1]}  \langle x^\star-x_i^\star, x^\star-x_m^\star\rangle- 
 \tfrac{1}{2}(\kappa-1)\\
& \sum_{i\in[m-1]} \left( \| x^\star-x^\star_i\|^2+ \| x^\star-x^\star_m\|^2 \right)\leq 0.
\end{align}
By the assumptions, we have for $i\in[m-1]$
\begin{align}\label{T3.2}
\|x^\star-x^\star_i\|^2+\|x^\star\|^2+2\langle x^\star-x^\star_i, x^\star\rangle-1\leq 0,
\end{align}
and
\begin{align}\label{T3.3}
\|x^\star-x^\star_m\|^2+\|x^\star\|^2+2\langle x^\star-x^\star_m, x^\star\rangle-1\leq 0.
\end{align}
 By multiplying inequalities \eqref{T3.1}, \eqref{T3.2} and \eqref{T3.3} by $\tfrac{1}{2m-2}(1+\tfrac{1}{\kappa})$, $\tfrac{1}{2m-2}(\kappa+1)$ and $\tfrac{1}{2}(\kappa+1)$, respectively, and summing them we obtain
\begin{align*}    
&  \|x^\star\|^2-(\kappa+1)+\tfrac{1}{4\kappa(m-1)}\sum_{i\in[m-1]}\big\| 2\kappa x^\star+(\kappa+1)(x^\star-
\\
 & \ x^\star_m)+(\kappa+1)(x^\star-x^\star_i)\big\|^2\leq 0.
 \end{align*}
 The above inequality implies the desired bound and the proof is complete. 
\end{proof}

\section{Conclusions}


 This work has demonstrated how the use of interpolation constraints simplifies the analysis of questions about convex functions and their minimizers. This framework can in principle allow for even more accurate descriptions of those sets of potential minimizers if additional information on the summands is provided, such as the bounds or exact values of the summands or their gradient at some remarkable points. 

\bibliographystyle{IEEEtran}
\bibliography{references} 

\end{document}